\documentclass[12pt,leqno]{article}
\usepackage{amsfonts}
\usepackage{amsmath, amsthm, amsfonts, amssymb, color}
\usepackage{mathrsfs}
\usepackage{color}
\newtheorem{thm}{Theorem}[section]
\newtheorem{cor}[thm]{Corollary}
\newtheorem{lem}[thm]{Lemma}

\newtheorem{exa}[thm]{Example}

\theoremstyle{definition}

\newcommand{\scr}[1]{\mathscr #1}
\definecolor{wco}{rgb}{0.5,0.2,0.3}

\numberwithin{equation}{section} \theoremstyle{remark}

\newcommand{\ua}{\uparrow}

\title{{\bf  Functional Inequalities for Convolution Probability  Measures}\footnote{Supported in
 part by  NNSFC(11131003, 11431014, 11201073), the 985-project and the Laboratory of Mathematical and  Complex Systems.}}
\author{
{\bf     Feng-Yu Wang$^{a), c)}$  and Jian Wang$^{b)}$}\\
\footnotesize{$^{a)}$School of Mathematical Sciences,
Beijing Normal
University, Beijing 100875, China}\\
\footnotesize{$^{b)}$School of Mathematics and Computer Science,
Fujian Normal University, Fuzhou 350007, China}\\
 \footnotesize{$^{c)}$Department of Mathematics,
Swansea University, Singleton Park, SA2 8PP, United Kingdom}\\
\footnotesize{  wangfy@bnu.edu.cn,
F.-Y.Wang@swansea.ac.uk,
 jianwang@fjnu.edu.cn}}
\begin{document}
\allowdisplaybreaks
\def\R{\mathbb R}  \def\ff{\frac} \def\ss{\sqrt} \def\B{\mathbf
B}
\def\N{\mathbb N} \def\kk{\kappa} \def\m{{\bf m}}
\def\ee{\varepsilon}
\def\dd{\delta} \def\DD{\Delta} \def\vv{\varepsilon} \def\rr{\rho}
\def\<{\langle} \def\>{\rangle} \def\GG{\Gamma} \def\gg{\gamma}
  \def\nn{\nabla} \def\pp{\partial} \def\E{\scr E}
\def\d{\text{\rm{d}}} \def\bb{\beta} \def\aa{\alpha} \def\D{\scr D}
  \def\si{\sigma} \def\ess{\text{\rm{ess}}}
\def\beg{\begin} \def\beq{\begin{equation}}  \def\F{\scr F}
\def\Ric{\text{\rm{Ric}}} \def\Hess{\text{\rm{Hess}}}
\def\e{\text{\rm{e}}} \def\ua{\underline a} \def\OO{\Omega}  \def\oo{\omega}
 \def\tt{\tilde} \def\Ric{\text{\rm{Ric}}}
\def\cut{\text{\rm{cut}}} \def\P{\mathbb P} \def\ifn{I_n(f^{\bigotimes n})}
\def\C{\scr C}      \def\aaa{\mathbf{r}}     \def\r{r}
\def\gap{\text{\rm{gap}}} \def\prr{\pi_{{\bf m},\varrho}}  \def\r{\mathbf r}
\def\Z{\mathbb Z} \def\vrr{\varrho} \def\ll{\lambda}
\def\L{\scr L}\def\Tt{\tt} \def\TT{\tt}\def\II{\mathbb I}
\def\i{{\rm in}}\def\Sect{{\rm Sect}}  \def\H{\mathbb H}
\def\M{\scr M}\def\Q{\mathbb Q} \def\texto{\text{o}} \def\LL{\Lambda}
\def\Rank{{\rm Rank}} \def\B{\scr B} \def\i{{\rm i}} \def\HR{\hat{\R}^d}
\def\to{\rightarrow}\def\l{\ell}\def\iint{\int}
\def\EE{\scr E}
\def\A{\scr A}
\def\BB{\scr B}

\maketitle

\begin{abstract} Let $\mu$ and $\nu$ be two probability measures on $\R^d$, where $\mu(\d x)= \frac{\e^{-V(x)}\d x}{\int_{\R^d} \e^{-V(x)}\d x}$ for some $V\in C^1(\R^d)$.  Explicit sufficient conditions on $V$ and   $\nu$   are presented such that
  $\mu*\nu$  satisfies the log-Sobolev,
 Poincar\'e and  super Poincar\'e   inequalities. In particular, if $V(x)=\ll|x|^2$ for some $\ll>0$
and $\nu(\e^{\ll\theta |\cdot|^2})<\infty$ for some $\theta>1$, then $\mu*\nu$ satisfies the log-Sobolev inequality. This improves and extends the recent results  on the log-Sobolev inequality derived in \cite{Z} for   convolutions of the  Gaussian measure and compactly supported probability measures. On
the other hand, it is well known that the log-Sobolev inequality for
$\mu*\nu$   implies $\nu(\e^{\vv|\cdot|^2}) <\infty$ for some
$\vv>0.$

\medskip

\centerline{\textbf{R\'esum\'e}}

\medskip

Soit $\mu$ et $\nu$ deux mesures de probabilit\'e sur $\R^d$, o\`u $\mu(\d x)= \frac{\e^{-V(x)}\d x}{\int_{\R^d} \e^{-V(x)}\d x}$ avec $V\in C^1(\R^d)$. Des conditions explicites suffisantes sur $V$ et $\nu$ sont pr\'esent\'ees telles que $\mu*\nu$ satisfait des in\'egalit\'es de Sobolev logarithmique, de Poincar\'e et de super-Poincar\'e. En particulier, si $V(x)=\ll|x|^2$ pour quelque $\ll>0$ et $\nu(\e^{\ll\theta |\cdot|^2})<\infty$ avec $\theta>1$, alors $\mu*\nu$ satisfait  l'in\'egalit\'e de Sobolev logarithmique. Cela am\'eliore et \'etend des r\'esultats r\'ecents sur l'in\'egalit\'e de Sobolev logarithmique obtenus dans \cite{Z} pour des convolutions de la mesure de Gauss et des mesures de probabilit\'e \`a support compact. D'autre part, il est bien connu que l'in\'egalit\'e de Sobolev logarithmique pour $\mu*\nu$  implique  $\nu(\e^{\vv|\cdot|^2}) <\infty$ pour quelque $\vv>0.$
\end{abstract}

 \vskip 2cm

\section{Introduction}\label{sec1}

Functional inequalities of Dirichlet forms are powerful tools in characterizing the properties of
 Markov semigroups and their generators, see e.g.\ \cite{Wbook} and references within. To establish functional inequalities for less explicit or less regular probability measures, one regards the measures as perturbations from better ones  satisfying  the underlying functional inequalities. For a probability measure $\mu$ on $\R^d$, the perturbation to  $\mu$ can be made in the following two senses. The first type  perturbation is in the sense of exponential potential: the perturbation  of $\mu$  by a potential $W$ is given by  $\mu_W(\d x):=\ff{\e^{W(x)}\mu(\d x)}{\mu(\e^W)}$, for which functional inequalities have been studied in many papers, see \cite{AS, BLW, CWW} and references within. Another kind of perturbation is in the sense of independent sum of random variables: the perturbation of $\mu$ by a probability measure $\nu$ on $\R^d$ is
 given by their convolution $$(\mu*\nu)(A):= \int_{\R^d} 1_{A}(x+y)\mu(\d x)\nu(\d y).$$ Functional inequalities for the latter case is not yet well investigated, and the study is useful in characterizing distribution properties of random variables under independent perturbations, see e.g. \cite[Section 3]{Z} for an application in the study of random matrices.

In general, let $\mu$ and $\nu$ be two probability measures on $\R^d$. A straightforward result on functional inequalities of $\mu*\nu$ can be derived from the sub-additivity property; that is, if both $\mu$ and $\nu$ satisfy a type of functional inequality, $\mu*\nu$ will satisfy the same type inequality. In this paper, we will consider the Poincar\'e inequality and the super Poincar\'e inequality. We say that a probability measure $\mu$ satisfies the Poincar\'e inequality with constant $C>0$, if
\begin{equation}\label{pin}\mu(f^2)\le C\mu(|\nn f|^2)+\mu(f)^2,\ \ f\in C_b^1(\R^d).\end{equation} We say that $\mu$ satisfies the super Poincar\'e inequality with $\beta: (0,\infty)\to (0,\infty)$, if
\begin{equation}\label{sup-pin}\mu(f^2)\le r \mu(|\nn f|^2) +\beta(r)\mu(|f|)^2,\ \ r>0, f\in C_b^1(\R^d).\end{equation}
It is shown in \cite[Corollary 3.3]{W00a} or \cite[Corollary 1.3]{W00b} that the super Poincar\'e inequality holds with
$\beta(r)=\e^{c/r}$ for some constant $c>0$ if and only if
  the following Gross log-Sobolev inequality  (see \cite{G}) holds for some constant $C>0$:
\begin{equation}\label{log}\mu(f^2\log f^2)\le C\mu(|\nn f|^2),\ \ f\in C_b^1(\R^d), \mu(f^2)=1.\end{equation}

\beg{prp}\label{P1.1} Let $\mu$ and $\nu$ be two probability measures on $\R^d$.\beg{enumerate}
\item[$(1)$] If $\mu$ and $\nu$ satisfy the Poincar\'e $($resp. log-Sobolev$)$ inequality with constants $C_1$ and $C_2>0$ respectively,
 then $\mu*\nu$ satisfies the same inequality with constant $C=C_1+C_2$.
    \item[$(2)$] If $\mu$ and $\nu$ satisfy the super Poincar\'e  inequality with   $\beta_1$ and $\beta_2$ respectively, then $\mu*\nu$ satisfies the super Poincar\'e inequality with   $$\beta(r):=\inf\big\{\beta_1(r_1)\beta_2(r_2): r_1,r_2>0, r_1+ r_2\beta_1(r_1)\le r\big\},\ \ r>0.$$
        \end{enumerate} \end{prp}

Since the proof of this result   is almost trivial by using
functional inequalities for product measures (cf.\  \cite[Corollary
13]{Ch}), we simply omit it. Due to Proposition \ref{P1.1}, in this
paper  the perturbation measure $\nu$ may not satisfy the
Poincar\'e inequality, it is in particular the case if the support
of $\nu$ is disconnected.

\

Recently, when $\mu$ is the Gaussian measure with variance matrix
$\dd I$ for some $\dd>0$, it is proved in  \cite{Z} that $\mu*\nu$
satisfies the log-Sobolev inequality if $\nu$ has a compact support
and  either $d=1$ or $\dd>2R^2d$, where $R$ is the radius of a ball
containing the support of $\nu$, see \cite[Theorem 2 and Theorem
17]{Z}. The first purpose of this paper is to extend this result to
more general $\mu$ and to drop the restriction $\dd>2R^2d$ for high
dimensions. The main tool used in \cite{Z} is the Hardy type
criterion for the log-Sobolev inequality due to \cite{BG}, which is
qualitatively sharp in dimension one. In this paper we will use a
perturbation result of \cite{AS} and  a Lyapunov type criterion
introduced in \cite{CGW}   to derive more general and  better
results. In particular, as a consequence of Corollary \ref{C1.3}
below, we have the following result where the compact support
condition of $\nu$ is relaxed by an exponential integrability
condition.   We would like to indicate that  the exponential
integrability condition $\nu(\e^{\vv |\cdot|^2})<\infty$ for some
$\vv>0$ is also necessary for $\mu*\nu$ to satisfy the log-Sobolev
inequality. Indeed, it is well known that the log-Sobolev inequality
for $\mu*\nu$ implies $(\mu*\nu)(\e^{c |\cdot|^2})<\infty$ for some
$c>0$, so that $\nu(\e^{\vv |\cdot|^2})<\infty$ for $\vv\in (0,c)$.
However, it is not clear whether $``\theta >1$" in the
following result is sharp or not.

\beg{thm}\label{T1.1} Let $V=\ll |\cdot|^2$ for some constant
$\ll>0$, and $\mu(\d x)= \frac{\e^{-V(x)}\d x}{\int_{\R^d} \e^{-V(x)}\d x}$ be a probability measure on $\R^d$. Then for any probability measure $\nu$ on $\R^d$ with
$\nu(\e^{\ll\theta|\cdot|^2})<\infty$ for some constant $\theta>1,$  the log-Sobolev inequality
$$(\mu*\nu)(f^2\log f^2)\le C(\mu*\nu)(|\nn f|^2),\ \ f\in C_b^1(\R^d), (\mu*\nu)(f^2)=1$$
holds for some constant $C>0.$ \end{thm}

According to the above-mentioned   results in \cite{Z}, one may wish to prove that the log-Sobolev inequality is stable
under convolution with compactly supported probability measures; i.e. if $\mu$ satisfies the log-Sobolev inequality, then so does $\mu*\nu$ for a probability measure $\nu$ having compact support. This is however not true, a simple counterexample is that $\mu=\dd_0,$ the Dirac measure at point $0$, which obviously satisfies the log-Sobolev inequality, but   $\mu*\nu=\nu$ does not have to satisfy the log-Sobolev inequality even if $\nu$ is compactly supported. Thus, to ensure that $\mu*\nu$ satisfies the log-Sobolev inequality for any compactly supported probability measure $\nu$, one needs additional assumptions on $\mu$. Moreover, since when $\ll\to\infty$, the Gaussian measure $\mu$ in Theorem \ref{T1.1} converges to $\delta_0$, this counterexample also fits to the assertion of Theorem \ref{T1.1} that for large $\lambda$ we need a stronger concentration condition on $\nu$.

In the remainder of this paper, let $\mu(\d x)=\e^{-V(x)}\d x$ be a
probability measure on $\R^d$ such that $V\in C^1(\R^d)$. For a probability measure $\nu$ on $\R^d$, we define
$$p_\nu(x)= \int_{\R^d} \e^{-V(x-z)}\nu(\d z),\ V_\nu(x)= -\log p_\nu(x), \ \ x\in \R^d.$$ Then
\beq\label{*0} (\mu*\nu)(\d x)= p_\nu(x)\d x = \e^{-V_\nu(x)}\d x.\end{equation}  Moreover, let
$$\nu_x(\d z)= \ff 1 {p_\nu(x)} \e^{-V(x-z)}\nu(\d z),\ \ x\in\R^d.$$
In the following three sections, we will investigate the log-Sobolev inequality, Poincar\'e   and super Poincar\'e inequalities for $\mu*\nu$ respectively.

As a complement to the present paper, Cheng and Zhang investigated the weak Poincar\'e inequality in \cite{chengzhang}
for convolution probability measures, by using the Lyapunov type conditions as we did in Sections 3 and 4 for the Poincar\'e and super Poincar\'e inequalities respectively.

\section{Log-Sobolev Inequality}
In this section we will use two different arguments to study the log-Sobolev inequality for $\mu*\nu$. One is the perturbation argument due to \cite{A,AS}, and the other is the Lyapunov criterion presented in \cite{CGW}.

\subsection{Perturbation Argument}

\beg{thm}\label{T1.2} Assume that the log-Sobolev inequality \eqref{log}
holds for $\mu$ with some constant $C>0$. If $V\in C^1(\R^d)$ such that
$$\Phi_\nu(x):= \int_{\R^d} (\nn \e^{-V})(x-z)\nu(\d z),\ \ x\in\R^d$$ is well-defined and continuous, and  there exists a constant
$\delta>1$ such that
\begin{equation}\label{T1.2.1} \int_{\R^d} \exp \left\{\frac{\delta C}{4}\Big(\int_{\R^d} |\nabla V(x)-\nabla V(x-z)|\nu_x(\d z)\Big)^2\right\}\, \mu (\d x)<\infty,\end{equation}  then  $\mu*\nu$ also
satisfies the log-Sobolev inequality, i.e. for some constant $C'>0$,
$$(\mu*\nu)(f^2\log f^2)\le C'(\mu*\nu)(|\nn f|^2),\ \ f\in C_b^1(\R^d), (\mu*\nu)(f^2)=1.$$
\end{thm}
Obviously, $\Phi_\nu\in C(\R^d;\R^d)$  holds if either $\nu$ has
compact support or $\nn\e^{-V}$ is bounded. Moreover, (\ref{C1.3.1})
below  holds for bounded $\Hess_V$ and compactly supported $\nu$.
So,  the following direct consequence of Theorem \ref{T1.2} improves
the above-mentioned main results in \cite{Z}. Indeed, this corollary
implies Theorem \ref{T1.1}.

\begin{cor}\label{C1.3} Assume that $(\ref{log})$ holds and $\Phi_\nu$ is well defined and continuous.
If $V\in C^2(\R^d)$ with bounded  $\Hess_V$
such that
\begin{equation}\label{C1.3.1} \int_{\R^d} \exp \left\{\frac{\delta C}{4}
\|\Hess_V\|_\infty^2\Big(\int_{\R^d} |z|\nu_x(\d z)\Big)^2\right\}\,
\mu (\d x)<\infty  \end{equation} holds for some constant $\dd>1$,
then $\mu*\nu$ satisfies the log-Sobolev inequality.   \end{cor}

Before presenting the proof of  Theorem \ref{T1.2}, we first prove Theorem \ref{T1.1}
using Corollary \ref{C1.3}.

 \beg{proof}[Proof of Theorem \ref{T1.1}]  Let $Z=\int_{\R^d}\e^{-\ll|x|^2}\d x.$  Since, in the framework of Corollary \ref{C1.3}, $V(x)=\ll |x|^2+\log Z$, we
have $\|\Hess_V\|_\infty^2=4\ll^2$ and \eqref{log} holds for $C=\ff
1 \ll$. Moreover, since $\theta>1$, there exists a constant $\vv\in (0,1)$ such that $\dd:= \theta-\ff\vv{1-\vv}>1.$ So, by the Jensen inequality \beg{equation}\label{L1}\beg{split} I&:=
\int_{\R^d}\exp\bigg\{\ff{\dd C}4
\|\Hess_V\|_\infty^2\nu_x(|\cdot|)^2\bigg\}\mu(\d x)\le \int_{\R^d} \e^{\dd \ll \nu_x(|\cdot|^2)}\mu(\d x)\\
&\le\int_{\R^d\times\R^d} \e^{\dd\ll|z|^2}\nu_x(\d z)\mu(\d x)
=\int_{\R^d} \ff{\int_{\R^d} \e^{\ll\dd |z|^2-\ll|x-z|^2}\nu(\d
z)}{\int_{\R^d} \e^{-\ll|x-z|^2}\nu(\d z)}\,\mu(\d x).\end{split}\end{equation}
Take   $R>0$ such that $\nu(B(0,R))\ge \ff 1 2$. We have
$$\int_{\R^d} \e^{-\ll |x-z|^2}\nu(\d z) \ge \int_{B(0,R)} \e^{-\ll R^2-2\ll R|x|-\ll |x|^2}\nu(\d z)\ge \ff 1 2 \e^{-\ll R^2-2\ll R|x|-\ll |x|^2}.$$ Moreover,
for the above $\vv\in (0,1)$ we have
$$-|x-z|^2 \le 2|x|\cdot |z|-|x|^2-|z|^2\le -\vv |x|^2 +\ff\vv{1-\vv} |z|^2.$$ Combining this with \eqref{L1}, we obtain
\beg{equation*}\beg{split} I&\le \ff {2\e^{\ll R^2}} Z \int_{\R^d\times\R^d} \e^{\ll\dd|z|^2-\ll|x-z|^2+2\ll R|x|}\nu(\d z)\d x\\
&\le  \ff {2\e^{\ll R^2}}Z \int_{\R^d\times\R^d} \e^{\ll\dd|z|^2-\ll \vv |x|^2 +\ff{\ll \vv}{1-\vv}|z|^2+2\ll R|x|}  \d x\nu(\d z)\\
&= \ff {2\e^{\ll R^2}}Z \int_{\R^d\times\R^d} \e^{\ll\theta|z|^2-\ll \vv |x|^2 +2\ll R|x|}  \d x\nu(\d z)<\infty.\end{split}\end{equation*}
  Then the proof is finished by Corollary \ref{C1.3}.
\end{proof}

To prove Theorem \ref{T1.2},
we introduce the following perturbation result  due to
\cite[Lemma 3.1]{AS} and \cite[Lemma 4.1]{A}.

\beg{lem}\label{L2.2}Assume that the probability measure $\mu(\d x)=\e^{-V(x)}\d x$ satisfies the
log-Sobolev inequality \eqref{log} with some constant $C>0$.   Let $\mu_{V_0}(\d
x)=\e^{-V_0(x)}\d x$ be a probability measure on $\R^d$. If $F:=\frac{1}{2}(V-V_0)\in C^1(\R^d)$ such that
\begin{equation}\label{l2.1} \int_{\R^d}\exp (\delta C|\nabla F|^2)\d\mu
<\infty,\end{equation} holds for some constant $\dd>1$, then the
defective log-Sobolev inequality
\begin{equation}\label{d-log}\mu_{V_0}(f^2\log f^2)\le C_1\mu_{V_0}(|\nn f|^2)+C_2,\ \ f\in C_b^1(\R^d), \mu_{V_0}(f^2)=1,\end{equation}holds for some constants $C_1,C_2>0.$
  \end{lem}

 \beg{proof}[Proof of Theorem \ref{T1.2}] Since by (\ref{*0}) we have $(\mu*\nu)(\d x)=\e^{-V_\nu(x)}\d x$, to apply Lemma \ref{L2.2} we take $V_0=V_\nu$, so that  $$F(x)=\frac{1}{2}(V(x)-V_0(x))=\frac{1}{2}\log \int_{\R^d} \e^{V(x)-V(x-z)}\nu(\d z).$$ Since $\Phi_\nu$ is locally bounded, for any $x\in \R^d$ we have
 $$\lim_{y\to 0} (p_\nu(x+y)-p_\nu(x))=\lim_{y\to 0} \int_0^1\<y,\Phi_\nu(x+sy)\>\d s=0.$$ So, $p_\nu\in C(\R^d)$. Then the continuity of $\Phi_\nu$ implies that
 $$\Psi(x):= \int_{\R^d}(\nn V)(x-z)\nu_x(\d z)=-\ff {\Phi_\nu(x)} {p_\nu(x)}$$ is continuous in $x$ as well.
 Therefore, for any $x,v\in\R^d$,
\beg{equation*}\beg{split} \lim_{\vv\downarrow 0} \ff{F(x+\vv v)-F(x)}\vv &= \lim_{\vv\downarrow 0} \ff 1 {2\vv} \int_0^\vv \<v, \nn V(x+sv)-\Psi(x+sv)\>\d s\\
& =\ff 1 2 \<v,\nn V(x)-\Psi(x)\>.\end{split}\end{equation*} Thus,  by the continuity of $\Psi$ and $\nn V$ we conclude that $F\in C^1(\R^d)$ and
$$ |\nabla F(x)|^2=\ff 1 4 |\nn V(x)-\Psi(x)|^2 \le \frac{1}{4}\Big(\int_{\R^d} |\nabla V(x)-\nabla V(x-z)|\nu_x(\d z)\Big)^2.$$ Combining this with (\ref{T1.2.1}), we are able to apply Lemma \ref{L2.2} to derive the defective log-Sobolev inequality for $\mu*\nu$. Moreover,   the form
  $$\E(f,g):= \int_{\R^d} \<\nn f,\nn g\> \d(\mu*\nu),\ \ f,g\in C_b^1(\R^d)$$ is closable in $L^2(\mu*\nu)$,
   and its closure is a symmetric, conservative, irreducible Dirichlet form. Thus, according to
   \cite[Corollary 1.3]{W13} (see also \cite[Theorem 1]{M2}), the defective log-Sobolev inequality implies the desired log-Sobolev inequality. Then the proof is finished.
 \end{proof}

To see that Corollary \ref{C1.3} has a broad range of application
beyond  \cite[Theorem 2]{Z} and Proposition \ref{P1.1}(1) for the
log-Sobolev inequality, we present below an example where the support of $\nu$ is
unbounded and disconnected.

\begin{exa}\label{E1}
Let $d=1$, $V(x)=\ff 1 2 \log\pi+x^2$ and
$$\nu(\d z)= \ff 1 \gg \sum_{i\in\Z} \e^{-\lambda i^2}\dd_i(\d z),\ \
\gg:=\sum_{i\in\Z}\e^{-\lambda i^2},$$ where $\dd_i$ is the Dirac measure at
point $i$ and $\ll>0$. Then $\mu*\nu$ satisfies the log-Sobolev inequality.
\end{exa}

\beg{proof}  For the present $V$ it is well known from \cite{G} that  the log-Sobolev
inequality \eqref{log} holds with $C=1$. On the other hand, it is
easy to see that for any $i\in \Z$, $x\in\R$ and $\ll >0$, we have
\beq\label{**1} |x-i|^2+\ll i^2= (1+\ll)\Big(i-\ff x{\ll+1}\Big)^2 +\ff{\ll
x^2}{1+\ll}.\end{equation}
Let $\tt p(x)= \sum_{i\in\Z}\e^{-(1+\ll)(i-x/(1+\ll))^2}.$ Then
\beq\label{**2} \nu_x(\d z)= \ff 1 {\gg(x)} \sum_{i\in\Z} \e^{-|x-i|^2-\ll i^2}\dd_i(\d z)=
 \ff  1 {\tt p(x)}\sum_{i\in \Z}\e^{-(1+\ll)(i-x/(1+\ll))^2}\dd_i(\d z),\end{equation}where $\gg(x)= \sum_{i\in\Z} \e^{-|x-i|^2-\ll i^2}.$
So,
\beg{equation*}\beg{split}\int_{\R^d} |z|\nu_x(\d z)&= \ff 1 {\tt p(x)}
\sum_{i\in\Z} |i|\e^{-(1+\ll)(i-x/(1+\ll))^2}\\
&\le \ff{|x|}{1+\ll} +\ff 1 {\tt p(x)}
\sum_{i\in\Z} \Big|i-\ff{x}{1+\ll}\Big|\e^{-(1+\ll)(i-x/(1+\ll))^2}\\
&\le \ff{|x|}{1+\ll}+c,\ \ x\in \R\end{split}\end{equation*}
holds for
\beq\label{NB}c:= \sup_{x\in [0, 1+\ll]} \ff 1{\tt p(x)} \sum_{i\in\Z} \Big|i-\ff{x}{1+\ll}\Big|\e^{-(1+\ll)(i-x/(1+\ll))^2}<\infty\end{equation} since the underlying function   is periodic with a period $[0,1+\ll]$.    Noting that $C=1$ and $\|\Hess_V\|^2 =4$, we conclude from this that   condition \eqref{C1.3.1} holds for $\dd\in (1,1+\ll).$ Then the proof is finished by Corollary \ref{C1.3}.
\end{proof}

Finally, the following example shows that Theorem \ref{T1.2} may also work  for unbounded  $\Hess_V$.

\begin{exa}\label{E2}
Let $V(x)=c +|x|^p$ with $p\in [2,4)$ for some constant $c$ such that $\mu(\d x):=\e^{-V(x)}\d x$ is a probability measure on $\R^d$. Let $\nu$ be a probability measure on $\R^d$ with compact support. Then $\mu*\nu$ satisfies the log-Sobolev inequality.
\end{exa}

\beg{proof}  Since $p\ge 2$, we have $V\in C^2(\R^d)$ and
$\Phi_\nu\in C(\R^d, \R^d).$ Let $R=\sup\{|z|: z\in {\rm
supp}\,\nu\}.$ Then
$$\int_{\R^d} |\nabla V(x)-\nabla  V(x-z)|\nu_x(\d z) \le R\sup_{z\in B(x,R)} |\Hess_{V}(z)|\le C(R)(1+|x|^{p-2})$$ holds for some constant $C(R)>0$ and all $x\in\R^d$. Combining this with $2(p-2)<p$ implied by $p<4$, we see that   (\ref{T1.2.1}) holds. Then the proof is finished by Theorem \ref{T1.2}.   \end{proof}

We will see in Remark 4.1 below that the assertion in Example
\ref{E2} remains true for $p\ge 4.$ Indeed, when $p>2$ the super
Poincar\'e inequality presented in Example \ref{E3.7} below is
stronger than the log-Sobolev inequality, see \cite[Corollary
3.3]{W00a}.

\subsection{Lyapunov Criterion}

\beg{thm} \label{T1.6} Assume that $V\in C^2(\R^d)$ with bounded $\Hess_V$ such that
\beq\label{T1.6.1} \Hess_V\ge K I \ \text{ \ outside\ a\ compact\
set}\end{equation} holds for some constant $K>0$. Then $\mu*\nu$
satisfies the log-Sobolev inequality provided the following two
conditions hold: \beg{enumerate}\item[$(C1)$] There exists a
constant $c>0$ such that $$\nu_x(f^2)-\nu_x(f)^2\le c\|\nn
f\|_\infty^2,\ \ f\in C_b^1(\R^d),x\in\R^d.$$
\item[$(C2)$] $\limsup_{|x|\to\infty} \dfrac{\int_{\R^d} |\nn V(-z)|\nu_x(\d z)}{|x|}<K.$\end{enumerate}
\end{thm}

We believe that Theorems \ref{T1.2} and \ref{T1.6} are incomparable, since (\ref{T1.6.1}) is neither necessary for (\ref{log}) to hold, nor provides explicit upper bound of $C$ in (\ref{log}) which is involved in condition (\ref{T1.2.1}) for Theorem \ref{T1.2}. But
   it would be rather complicated to construct proper counterexamples confirming this observation.

The proof of Theorem \ref{T1.6}  is based on
the following Lyapunov type criterion due to \cite[Theorem
1.2]{CGW}.

\beg{lem}[\cite{CGW}]\label{L2.1} Let $\mu_0(\d x)=\e^{-V_0(x)}\d x$
be a probability measure on $\R^d$ for some $V_0\in C^2(\R^d).$ Then
$\mu_0$ satisfies the log-Sobolev inequality provided the following
two conditions hold: \beg{enumerate}\item[{\rm (i)}] There exists a
constant $K_0\in\R$ such that $\Hess_{V_0}\ge K_0I$.
\item[{\rm (ii)}] There exists $W\in C^2(\R^d)$ with $W\ge 1$ such that
$$\DD W(x)-\<\nn V_0,\nn W\>(x)\le (c_1-c_2|x|^2)W(x),\ \ x\in\R^d$$
holds for some constants $c_1,c_2>0$.\end{enumerate}
\end{lem}

\beg{proof}[Proof of Theorem \ref{T1.6}]  By (\ref{*0}) and  Lemma \ref{L2.1}, it suffices to verify conditions
(i) and (ii) for $V_0=V_{\nu}:=-\log p_\nu$.

(a) Proof of (i). By the boundedness of $\Hess_V$ and the condition
(\ref{T1.6.1}), it is to see that $p_\nu\in C^2(\R^d)$ and for any
$X\in \R^d$ with $|X|=1$, we have
\beg{equation}\label{2.3}\Hess_{V_0}(X,X)= \ff 1
{p_\nu^2}\Big((\nn_Xp)^2
-p_\nu\Hess_{p_\nu}(X,X)\Big).\end{equation} Moreover,
$$\nn_Xp_\nu(x)= -p_\nu(x)\int_{\R^d} (\nn_XV(x-z))\nu_x(\d z).$$ Then, letting
$K_1:=\|\Hess_V\|<\infty$, we obtain  \beg{equation*}\beg{split} \Hess_{p_\nu}(X,X)(x)
&=\int_{\R^d}
\Big(|\nn_XV(x-z)|^2-\Hess_V(X,X)(x-z)\Big)\e^{-V(x-z)}\nu(\d z)\\
&\le p_\nu(x) \int_{\R^d} |\nn_XV(x-z)|^2\nu_x(\d z)
+K_1p_\nu(x).\end{split}\end{equation*}Combining these with (\ref{2.3})
and $(C1)$,   we conclude that \beg{equation*}\beg{split}
\Hess_{V_0}(X,X)(x) &\ge -K_1 -\int_{\R^d}
(\nn_XV(x-z))^2\nu_x(\d z)  +\bigg(\int_{\R^d} \nn_XV(x-z)
\nu_x(\d z)\bigg)^2\\
&\ge -K_1-cK_1^2.\end{split}\end{equation*} Thus, (i) holds for
 $K_0=-K_1-cK_1^2$.

 (b) Proof of (ii). Let $W(x)=\e^{\vv |x|^2}$ for some constant $\vv>0.$ Then
\beq\label{2.4}  \ff{\DD W-\<\nn V_0,\nn W\>}{W}(x) = 2d\vv
+4\vv^2|x|^2 -\vv  \int_{\R^d} \<x,\nn V(x-z)\>\nu_x(\d
z).\end{equation} Since $\Hess_V$ is bounded and \eqref{T1.6.1}
holds, we know that  $\int_{\R^d} \<x,\nn V(x-z)\>\nu_x(\d z)$ is
well defined and locally bounded. By (\ref{T1.6.1}), there exists a
constant $r_0>0$ such that $\Hess_V\ge KI$ holds on the set
$\{|z|\ge r_0\}$. So, for $x\in\R^d$ with $|x|>2r_0$,
\beg{equation*}\beg{split} \<\nn V(x-z)-\nn V(-z),x\>
&=|x|\int_0^{|x|}\Hess_V\Big(\ff{x}{|x|},\ff{x}{|x|}\Big)\Big(\ff{rx}{|x|}-z\Big)\d
r\\
&\ge K|x|^2 -K_1|x|\Big|\Big\{r\in [0,|x|]: \Big|\ff{r
x}{|x|}-z\Big|\le r_0\Big\}\Big|\\&\ge K|x|^2 - 2 K_1 r_0
|x|.\end{split}\end{equation*} Combining this with (\ref{2.4}) and
$(C2)$, and noting that  $$ \<x,\nn V(x-z)\>\le\<\nn V(x-z)-\nn
V(-z),x\> +|x|\cdot |\nn V(-z)|,$$ we conclude that there exist  constants
$C_1,C_2>0$ such that $$\ff{\DD W-\<\nn V_0,\nn W\>}{W}(x) \le 2d\vv
+4\vv^2|x|^2-\vv C_1|x|^2 +\vv C_2.$$ Taking $\vv=\ff {C_1}8$, we
prove (ii) for some constants $c_1,c_2>0.$
\end{proof}

Since when $\nu$ has compact support, we have
$$\nu_x(f^2)-\nu_x(f)^2=\int_{\R^d\times\R^d} |f(z)-f(y)|^2\nu_x(\d
z)\nu_x(\d y)\le  R^2\|\nn f\|_\infty^2,$$ where $R:=\sup\{|z-y|:\
z,y\in {\rm supp}\nu\}<\infty$, and
$$\lim_{|x|\to\infty} \ff{\int_{\R^d} |\nn V(-z)|\nu_x(\d z)}{|x|}\le \lim_{|x|\to\infty} \ff{\sup_{{\rm supp}\nu} |\nn
V|}{|x|}=0,$$  The following direct consequence of Theorem
\ref{T1.6}   improves the above mentioned  results in \cite{Z}
as well.

\beg{cor}\label{C1.1} Assume that $V\in C^2(\R^d)$ with bounded $\Hess_V$ such that  $(\ref{T1.6.1})$ holds. Then $\mu*\nu$ satisfies the log-Sobolev inequality for any compactly supported probability measure $\nu$.\end{cor}

To show that Theorem \ref{T1.6} also has a range of application
beyond Corollary \ref{C1.1} and Proposition \ref{P1.1}(1) for the
log-Sobolev inequality, we reprove Example
\ref{E1} by using Theorem \ref{T1.6}.

 \beg{proof}[Proof of Example \ref{E1}   using Theorem \ref{T1.6}] Obviously, (\ref{T1.6.1}) holds for $K=2$.
Let
$$ \tt \nu_{x}=\ff1 {\tt\gg(x) } \sum_{i\in\Z} \e^{-(1+\ll)(i-x)^2}\dd_i,\ \ \tt\gg(x)= \sum_{i\in\Z} \e^{-(1+\ll)(i-x)^2}.$$
By (\ref{**1}) we have $\tt\nu_{x}=\nu_{(1+\ll)x}.$   Thus, we only need to
verify conditions $(C1)$ and $(C2)$ for $\tt\nu_x$ in place of
$\nu_x$.

(a) To prove condition $(C1)$, we make use of a Hardy type
inequality for birth-death processes with Dirichlet boundary
introduced in \cite{M}. Let $x\in \R$ be fixed. For any bounded
function $f$ on $\Z$, let $\tt f(i)= f(i)-f(i_x)$, where
$i_x:=\sup\{i\in\Z: i\le x\}$ is the integer part of $x$. Then
\beg{equation}\label{2.5} \tt\nu_x(f^2)-\tt\nu_x(f)^2\le
\sum_{i=-\infty}^{i_x} \tt f(i)^2\tt\nu_x(i) + \sum_{i=i_x}^\infty
\tt f(i)^2\tt\nu_x(i).\end{equation} It is easy to see that there
exists a constant $c>0$ independent of $x$ such that for any $m\ge
i_x>x-1$,
$$\sum_{i=i_x}^m \e^{(1+\ll)(i-x)^2}\le c\e^{(1+\ll)(m-x)^2},\ \
\sum_{i=m+1}^\infty \e^{-(1+\ll)(i-x)^2}\le c
\e^{-(1+\ll)(m+1-x)^2}.$$ Therefore, \beg{equation*}\beg{split}
&\sup_{m\ge i_x} \Big(\sum_{i=i_x}^m
\e^{(1+\ll)(i-x)^2}\Big)\sum_{i=m+1}^\infty
\e^{-(1+\ll)(i-x)^2}\\
&\le c^2\e^{(1+\ll)\{(m-x)^2-(m+1-x)^2\}}
 = c^2
\e^{(1+\ll)\{2(x-m)-1\}}\le c^2\e^{1+\ll}.\end{split}\end{equation*}
By this and the Hardy inequality (see   \cite[Theorem
1.3.9]{Wbook}), we have
$$\sum_{i=i_x}^\infty \tt f(i)^2\tt\nu_x(i)\le 4 c^2\e^{1+\ll}\sum_{i=i_x}^\infty
(f(i+1)-f(i))^2\tt\nu_x(i).$$ Similarly,
$$\sum_{i=-\infty}^{i_x} \tt f(i)^2\tt\nu_x(i)\le 4
c^2\e^{1+\ll}\sum_{i=-\infty}^{i_x} (f(i-1)-f(i))^2\tt\nu_x(i).$$
Combining these with (\ref{2.5}) we prove $(C1)$ for $\tt\nu_x$ and
some constant $c>0$ (independent of $x\in \R$).

(b) Let $\tt p(x)= \sum_{i\in\Z}\e^{-(1+\ll)(i-x/(1+\ll))^2}.$
 Noting that $\nn V(z)= 2 z$, by (\ref{**2}) we obtain
\beg{equation*}\beg{split}\int_{\R^d}|\nn V(-z)|\,\nu_x(\d z)&= \ff 2 {\tt p(x)}
\sum_{i\in\Z} |i|\e^{-(1+\ll)(i-x/(1+\ll))^2}\\
&\le \ff{2|x|}{1+\ll} +\ff 2 {\tt p(x)}
\sum_{i\in\Z} \Big|i-\ff{x}{1+\ll}\Big|\e^{-(1+\ll)(i-x/(1+\ll))^2}\\
&\le c+ \ff{2|x|}{1+\ll}\end{split}\end{equation*}
for    $c>0$ in \eqref{NB}. Therefore,
$$\limsup_{|x|\to\infty} \ff{\int_{\R^d}|\nn V(-z)|\,\nu_x(\d z)}{|x|}\le \ff 2 {1+\ll} <2=K.$$ Thus, condition $(C2)$ holds.
\end{proof}

At the end of this section, we present the following two remarks for perturbation argument and Lyapunov criteria to deal with convolution probability measures.
\paragraph{Remark 2.1} (1) Both Theorems \ref{T1.2} and \ref{T1.6} are concerned with qualitative conditions ensuring the existence of the log-Sobolev inequality for convolution probability measures.  It would be interesting to derive explicit estimates on the log-Sobolev constant, i.e. the smallest constant such that the log-Sobolev inequality holds.  Recently, by using refining the conditions in Lemma \ref{L2.1},  Zimmermann has estimated the log-Sobolev constant in \cite{zi2014} for the convolution  of a Gaussian measure with a compactly supported measure (see \cite[Theorem 10]{zi2014} for more details). Similar things can be done under the present general framework.  However, as it is well known that estimates derived from perturbation arguments are in general less sharp,   we will not go further in this direction and leave the  quantitative estimates to a forthcoming paper by other means.

(2) As mentioned in Section \ref{sec1}, the convolution of probability measures refers to the sum of independent random variables. So, by induction we may use the Lyapunov criteria to investigate functional inequalities for multi-convolution measures. In this case it is  interesting   to study the behavior  of the optimal constant  (e.g. the log-Sobolev constant) as multiplicity goes to $\infty$. For this we need fine estimates on the constant in terms of the multiplicity, which is related to what we have discussed in Remark 2.1(1).    Of course, for functional inequalities having the sub-additivity property, it is possible to derive multiplicity-free estimates on the optimal constant, see e.g. the recent paper \cite{PLS} for Beckner-type inequalities of convolution measures on the abstract Wiener space.

\section{Poincar\'{e}  inequality}\label{sec3}

In the spirit of the proof of Theorem \ref{T1.6}, in this section we
study the  Poincar\'e
inequality for convolution measures  using the
Lyapunov conditions presented in \cite{BCG, BBCG}. One may also wish to use the following easy to check perturbation result on
 the Poincar\'e inequality corresponding to Lemma \ref{L2.2}.

 \emph{If the probability measure $\mu_V(\d x)=e^{-V(x)}\d x$ satisfies
 the Poincar\'e inequality \eqref{pin} with some constant $C>0$, then for any $V_0\in C^1(\R^d)$ such
 that $\int e^{-V_0(x)}\d x=1$ and
 $C \|\nn (V-V_0)\|_\infty^2<2$, the probability measure $\mu_{V_0}(\d x)=e^{-V_0(x)}\d x$
  satisfies the Poincar\'e inequality \eqref{pin} (with a different constant) as well.}

  Since
 the boundedness condition on $\nn(V-V_0)$ is rather strong (for instance, it excludes Example \ref{E3.3}(1) below for $p>2$), here, and also in the next section for the super Poincar\'e
 inequality, we will  use the Lyapunov criteria rather than  this perturbation result. By combining the following Theorem \ref{T3.1} below with \cite[Theorem 1.4]{BBCG}, one may derive   quantitative estimates on the  Poincar\'{e} constant (or the spectral gap).

\beg{thm}\label{T3.1} Let $\mu(\d x)=\e^{-V(x)}\d x$ be a
probability measure on $\R^d$ and let $\nu$ be a probability measure
on $\R^d$. Assume that $\Phi_\nu$ in Theorem $\ref{T1.2}$ is well-defined and continuous. Then $\mu*\nu$ satisfies the Poincar\'e inequality
\eqref{pin}, if at least one of the following conditions holds:
\beg{enumerate}\item[$(1)$] $V\in C^1(\R^d)$ such that
 $\liminf\limits_{|x|\to\infty} \dfrac{\int_{\R^d} \langle x, \nabla V(x-z)\rangle \nu_x(\d
z)}{|x|}>0.$ \item[$(2)$] $V\in C^2(\R^d)$ such that $\tt\Phi_\nu(x):= \int_{\R^d}(\nn^2 V)(x-z) \nu_x(\d z)$ is well-defined and continuous in $x$, and  there is a constant $\delta\in(0,1)$ such that
$$ \liminf_{|x|\to\infty}  \int_{\R^d}\! \Big(\delta|\nabla V(x-z)|^2-\triangle V(x-z)\Big)
\, \nu_x(\d z)>0.$$
\end{enumerate} \end{thm}

\begin{proof}  Let $L_{\nu}=\Delta-\nabla V_\nu$. According to \cite[Theorem 3.5]{BCG} or \cite[Theorem
1.4]{BBCG}, $(\mu*\nu)(\d x):=\e^{-V_\nu(x)}\d x$ satisfies the
Poincar\'{e} inequality if there exist a   $C^2$-function $W\ge1$ and some
positive constants $\theta, b, R$ such that for all $x\in \R^d$,
\begin{equation}\label{pi-1}L_{\nu}W(x)\le -\theta W(x)+b1_{B(0,R)}(x).\end{equation} In particular, by
\cite[Corollary 1.6]{BBCG},  
if either
 \beq\label{*D1}\liminf_{|x|\to\infty} \ff{\<\nn V_\nu(x),x\>}{|x|}>0,\end{equation} or
  there is a constant $\delta\in (0,1)$ such that
\beq\label{*D2}\liminf\limits_{|x|\to\infty} \Big(\delta |\nabla V_\nu(x)|^2-\DD
V_\nu(x)\Big)>0,\end{equation}
then the inequality \eqref{pi-1} fulfills.

Now, as shown in the proof of Theorem \ref{T1.2} that the continuity of $\Phi_\nu$ implies that $V_\nu\in C^1(\R^d)$   and    $$\<\nn
V_\nu(x),x\>=\int_{\R^d}\<\nabla V(x-z), x\>\nu_x(\d z).$$ Then condition (1) in Theorem \ref{T3.1} implies (\ref{*D1}), and hence the Poincar\'e inequality for $\mu*\nu$.

On the other hand, repeating the argument leading to $F\in C^1(\R^d)$ in the proof of Theorem \ref{T1.2},
we conclude that the continuity of $\Phi_\nu$ and $\tt\Phi_\nu$ implies   $  V_\nu\in C^2(\R^d)$   and
\beg{equation*}\beg{split} &|\nabla V_\nu(x)|^2=\left(\int_{\R^d} \nabla V(x-z)\nu_x(\d z)\right)^2,\\
&\DD V_\nu(x)=|\nabla V_\nu(x)|^2+\int_{\R^d} \big\{\DD V(x-z) -
|\nabla V(x-z)|^2\big\}\nu_x(\d z).\end{split}\end{equation*} Then  for any $\delta\in(0,1)$,
\beg{equation*}\beg{split}\delta |\nabla V_\nu(x)|^2-\Delta V_\nu(x)=&\int_{\R^d} \Big(|\nabla V(x-z)|^2-\DD V(x-z)\Big)\,
\nu_x(\d z)-(1-\delta)|\nabla V_\nu(x)|^2\\
\ge& \int_{\R^d} \Big(\delta|\nabla V(x-z)|^2-\DD V(x-z)\Big)\,
\nu_x(\d z).\end{split}\end{equation*} Combining this with condition (2) in  Theorem \ref{T3.1} we prove (\ref{*D2}), and hence the Poincar\'e inequality for $\mu*\nu$.
\end{proof}

When the measure $\nu$ is  compactly supported, we
have the following consequence of Theorem \ref{T3.1}.

\begin{cor}\label{C3.2} Let $\nu$ be a probability measure
on $\R^d$ with compact support such that $R:=\sup\{|z|: z\in {\rm
supp}\,\nu\}<\infty.$ If either $V\in C^1(\R^d)$ with
  \begin{equation}\label{C3.2.1}\liminf_{|x|\to\infty}
\ff{\<\nn V(x),x\>-R|\nn V(x)|}{|x|}>0,\end{equation}
or $V\in C^2(\R^d)$ and there is a constant $\delta\in(0,1)$ such that
  \begin{equation}\label{C3.2.2} \liminf_{|x|\to\infty}  \big(\delta|\nabla V(x)|^2-\DD V(x)\big)>0,\end{equation}
then $\mu*\nu$ satisfies the Poincar\'e inequality.\end{cor}

\begin{proof} Since the support of $\nu$ is compact,  the continuity of $\Phi_\nu$ when $V\in C^1(\R^d)$ and that of  $\tt\Phi_\nu$ when   $V\in C^2(\R^d)$ are obvious. Below we prove conditions (1) and (2) in Theorem \ref{T3.1} using (\ref{C3.2.1}) and (\ref{C3.2.2}) respectively.

(a) By \eqref{C3.2.1}  we obtain \beg{equation*}\beg{split} {\int_{\R^d}
\langle x, \nabla V(x-z)\rangle \nu_x(\d z)}&=\int_{\R^d}
\Big(\<x-z,\nn V(x-z)\> +\<z, \nn V(x-z)\>\Big)\nu_x(\d z)\\
&\ge \int_{\R^d} \Big(\<x-z,\nn V(x-z)\> -R|\nn
V(x-z)|\Big)\nu_x(\d z)\\
 &\ge \int_{\R^d} \big(c_1|x-z|-c_2\big)\nu_x(\d z) \\
 &\ge
c_1(|x|-R)^+ -c_2\end{split}\end{equation*}
 for some constants $c_1,c_2>0$. Then  condition (1) in Theorem \ref{T3.1} holds.

 (b) According to \eqref{C3.2.2}, there are constants $r_1, c_3$ and $c_4>0$
 such that for all $x\in\R^d$
\beg{equation}\label{D3}\beg{split}&\int_{\R^d} \Big(\delta|\nabla
V(x-z)|^2-\DD V(x-z)\Big) \, \nu_x(\d z)\\
&\ge
c_3\int_{\{|x-z|>r_1\}}\, \nu_x(\d z)-c_4\int_{\{|x-z|\le r_1\}}\,
\nu_x(\d z).\end{split}\end{equation} Since for $x\in\R^d$
with $|x|>R+r_1$ we have  $$\int_{\{|x-z|>r_1\}}\, \nu_x(\d z)\ge \int_{\{
|z|\le R\}}\, \nu_x(\d z)=1 $$ and
$$\int_{\{|x-z|\le r_1\}}\, \nu_x(\d z)\le \int_{\{|z|>R\}}\, \nu_x(\d
z)=0,$$   then (\ref{D3}) implies condition (2) in Theorem \ref{T3.1}.
\end{proof}

Finally, we present the following examples to illustrate  Theorem
\ref{T3.1} and Corollary \ref{C3.2}.

\begin{exa}\label{E3.3} $(1)$ Let $V(x)= c+|x|^p$ for some $p\ge 1$ and constant $c$ such that $\mu(\d x):= \e^{-V(x)}\d x$ is a
probability measure on $\R^d$.  Then   $\mu*\nu$ satisfies the Poincar\'{e} inequality for every compactly supported probability measure $\nu$ on $\R^d$.

$(2)$ Let $d=1$, $V(x)=c+ \sqrt{1+x^2}$ and
$$\nu(\d z)= \ff 1 \gg \sum_{i\in\Z} \e^{-|i|}\dd_i(\d z),\ \
\gg:=\sum_{i\in\Z}\e^{-|i|},$$ where $c=\log \int_{\R} \e^{-\sqrt{1+x^2}}\d x$ and  $\dd_i$ is the Dirac measure at
point $i$. Then $\mu*\nu$ satisfies the Poincar\'{e} inequality.
\end{exa}

\begin{proof}  Since when $p<2$ the function $V(x)=c+|x|^p$ is not in $C^2$ at point $0$, we take $\tilde{V}\in C^2(\R^d)$ such that
$\tilde{V}(x)=V(x)$ for $|x|\ge1$. Let $\tt\mu (\d x)= \tt C \e^{-\tt V(x)}\d x$, where $\tt C>0$ is a constant such that $\tt\mu$ is a probability measure. By the
stability of Poincar\'{e} inequality under the bounded perturbations
(e.g.\ see \cite[Proposition 17]{Ch}), it suffices to prove that $\tt\mu*\nu$ satisfies the Poincar\'e inequality.

In case (1) the assertion is  a direct consequence
of Corollary \ref{C3.2}. So, we only have to verify condition (1) in Theorem \ref{T3.1} for case (2).   For simplicity, we only
verify for  $x\to\infty,$  i.e.
\beq\label{3.1'} \lim_{x\to\infty} \frac{\int_{\R} x V'(x-z)\nu_x(\d
z)}{|x|} >0.\end{equation}
 Let $i_x$ be the integer
part of $x$, and $h_x=x-i_x.$ Note that for any $x>0$,
\beg{equation}\label{3.2'} \beg{split} \frac{\int_{\R} x V'(x-z)\nu_x(\d
z)}{|x|}&=\int_{\R} V' (x-z)\nu_x(\d z)\\
&=\frac{\sum_{i\in
\Z}\frac{x-i}{\sqrt{1+(x-i)^2}}\e^{-\sqrt{1+(x-i)^2}-|i|}}
{\sum_{i\in \Z} \e^{-\sqrt{1+(x-i)^2}-|i|}}\\
&=\frac{\sum_{k\in \Z}\frac{h_x+k}{\sqrt{1+(h_x+k)^2}}\e^{-\sqrt{1+(h_x+k)^2}-|i_x-k|}}
{\sum_{k\in \Z} \e^{-\sqrt{1+(h_x+k)^2}-|i_x-k|}}\\
&=:1-p_\nu(x)^{-1}\sum_{k\in\Z} (a_k b_k)(x),\end{split}\end{equation}
where \beg{equation*}\beg{split} &a_k(x):=\frac{\sqrt{1+(h_x+k)^2}-(h_x+k)}{\sqrt{1+(h_x+k)^2}},\\
&b_k(x):={\e^{-\sqrt{1+(h_x+k)^2}-|i_x-k|}},\ \
 p_\nu(x)= \sum_{k\in \Z} b_k(x).\end{split}\end{equation*} It is easy to see that
  $$0\le a_k(x)\le \beg{cases}  (1+k^2)^{-1/2}, &  k\ge 0,\\
  2, & k<0.\end{cases}$$ Then for any $n\ge 1$,
\beg{equation*}\beg{split} \sum_{k\in\Z} (a_k b_k)(x)&=\sum_{k\le 0}(a_k b_k)(x)+\sum_{k=1}^n(a_k b_k)(x)+ \sum_{k=n+1}^\infty a_k b_k(x)\\
&\le 2\sum_{k\le 0}b_k(x)+\sum_{k=1}^nb_k(x)+\ff 1 {n+1}\sum_{k=n+1}^\infty   b_k(x). \end{split}
\end{equation*}Thus,   for any $x>0$ and $1\le n\le i_x$,
\beg{equation*}\beg{split}
&\sum_{k\le 0}b_k(x)\le \e^{-x}+\sum_{k=-\infty}^{-1} \e^{-(-k-h_x)-(i_x-k)}
\le (2\e^2+1)\e^{-x},\\
&\sum_{k=1}^nb_k(x)\le n\e^{-x},\ \
 p_\nu(x)\ge \sum_{k=1}^{i_x}b_k(x)\ge i_x\e^{-x-1}.\end{split}\end{equation*} Then for any $n\ge 1$,
  $$\limsup_{x\to\infty} \ff 1 {p_\nu(x)} \sum_{k\in\Z} (a_kb_k)(x) \le \lim_{x\to\infty} \Big\{\ff{\e^{x+1}(2\e^2+1+n)\e^{-x}}{i_x}+\ff 1 {n+1}\Big\}=\ff 1{n+1}.$$ Letting $n\to\infty$ we obtain
$\lim_{x\to\infty} p_\nu(x)^{-1}\sum_{k\in\Z} (a_k b_k)(x)=0$. Combining this with (\ref{3.2'}) we prove  (\ref{3.1'}).
\end{proof}

\section{Super Poincar\'{e} Inequality}\label{sec4}

In this section we extend the results in Section 3 for the super Poincar\'e inequality.

\beg{thm}\label{T3.4} Let $\mu(\d x)=\e^{-V(x)}\d x$ be a
probability measure on $\R^d$ and let $\nu$ be a probability measure
on $\R^d$. Define $$\alpha(r,s)=(1+s^{-d/2})\frac{\Big(\sup_{|x|\le
r}\e^{-V(x)}\Big)^{d/2+1}}{\Big(\inf_{|x|\le
r}\e^{-V(x)}\Big)^{d/2+2}},\ \ s,r>0.$$
\beg{enumerate}
\item[$(1)$] If $V\in C^1(\R^d)$ such that
\begin{equation}\label{T3.4.1}\liminf_{|x|\to\infty} \frac{\int_{\R^d} \langle x, \nabla V(x-z)\rangle \nu_x(\d
z)}{|x|}=\infty, \end{equation} then $\mu*\nu$ satisfies the super
Poincar\'{e} inequality \eqref{sup-pin} with
$$\beta(r)=c\Big(1+\alpha(\psi(2/r), r/2)\Big)$$ for some constant $c>0$, where
$$\psi(r):=\inf\bigg\{s>0: \inf_{|x|\ge s}\frac{\int_{\R^d} \langle x, \nabla V(x-z)\rangle \nu_x(\d
z)}{|x|}\ge r\bigg\}<\infty,\ \ r>0.$$
\item[$(2)$] Suppose that $V\in C^2(\R^d)$ and there is a constant $\delta\in(0,1)$ such that
\begin{equation}\label{T3.4.2} \liminf_{|x|\to\infty}  \int_{\R^d}\! \Big(\delta|\nabla V(x-z)|^2-\DD V(x-z)\Big)
\, \nu_x(\d z)=\infty.\end{equation} Then, $\mu*\nu$ satisfies the
super Poincar\'{e} inequality \eqref{sup-pin} with
$$\beta(r)=c\Big(1+\alpha(\tt\psi (2/r), r/2)\Big)$$ for some constant $c>0$, where
$$\tt\psi(r):=\inf\bigg\{s>0: \inf_{|x|\ge s}
\int_{\R^d}\! \Big(\delta|\nabla V(x-z)|^2-\DD V(x-z)\Big)\, \nu_x(\d z)\ge r\bigg\}<\infty,\ \ r>0.$$
\end{enumerate}\end{thm}

The proof of Theorem \ref{T3.4} is based on the following lemma.

\begin{lem}\label{L3.5}Let $\mu(\d x)=\e^{-V(x)}\d x$ be a probability measure on $\R^d$.
Assume that there are functions $W\ge 1$, $\phi>0$ with
$\liminf_{|x|\to\infty}\phi(x)=\infty$ and constants $b,r_0>0$ such
that
$$\frac{\DD W- \<\nabla W, \nabla V\>}{W}\le
-\phi+b1_{B(0,r_0)}.$$ Then, the following super Poincar\'{e}
inequality holds
$$\mu_V(f^2)\le r \mu_V(|\nabla f |^2)+\beta(r) \mu_V(|f|)^2,$$
with   $$\beta(r)=c \Big(1+\alpha(\psi_\phi(2/r), r/2)\Big),\ \
r>0$$ for some constant $c >0$ and $$\psi_\phi(r):=\inf\big\{s>0:
\inf_{|x|\ge s}\phi(x)\ge r\big\}.$$
\end{lem}

\begin{proof} It is well known that (e.g.\ see \cite[Proposition 3.1]{CGWW}) there exists a constant $C>0$ such that
for any $t, s>0$ and $f\in C^1(\R^d)$,
$$\int_{B(0,t)}f^2(x)\d x\le s \int_{B(0,t)}|\nabla f(x)|^2\d x+ C(1+s^{-d/2})\Big(\int_{B(0,t)}|f|(x)\d x\Big)^2.$$ Therefore,
\beg{equation*}\beg{split} \int_{B(0,t)}f^2(x)\mu_V(\d x)&\le
\Big(\sup_{|x|\le t}\e^{-V(x)}\Big)\int_{B(0,t)}f^2(x)\d x\\
&\le s \frac{\sup_{|x|\le t}\e^{-V(x)}}{\inf_{|x|\le t}\e^{-V(x)}
}\int_{B(0,t)}|\nabla f(x)|^2\mu_V(\d x)\\
&\quad + C(1+s^{-d/2}) \frac{\sup_{|x|\le
t}\e^{-V(x)}}{\Big(\inf_{|x|\le
t}\e^{-V(x)}\Big)^2}\Big(\int_{B(0,t)}|f|(x)\mu_V(\d x)\Big)^2\\
&\le s \frac{\sup_{|x|\le t}\e^{-V(x)}}{\inf_{|x|\le t}\e^{-V(x)} }
\mu_V(|\nabla f|^2)+C(1+s^{-d/2}) \frac{\sup_{|x|\le
t}\e^{-V(x)}}{\Big(\inf_{|x|\le t}\e^{-V(x)}\Big)^2}\mu_V(|f|)^2.
\end{split}\end{equation*}
Taking  $s=r\frac{\inf_{|x|\le t}\e^{-V(x)} }{\sup_{|x|\le
t}\e^{-V(x)}}$ in the inequality above, we arrive at that for any $t,
r>0$ and $f\in C^1(\R^d)$,
 $$\int_{B(0,t)}f^2(x)\mu_V(\d x)\le r \mu_V(|\nabla
 f|^2)+C\alpha(t,r)\mu_V(|f|)^2.$$
Thus, the proof is finished by \cite[Theorem 2.10]{CGWW} and
the fact that the function $\alpha(r,s)$ is increasing with respect
to $r$ and decreasing with respect to $s$.
\end{proof}

\begin{proof}[Proof of Theorem \ref{T3.4}] As the same to the proof of Theorem \ref{T3.1}, let $L_{\nu}=\Delta-\nabla V_\nu$.

In case (1), we consider a smooth function such that $W(x)=\e^{2|x|}$ for $|x|\ge1$ and $W(x)\ge 1$ for all $x\in \R^d$. We have
$$\frac{L_{\nu}W(x)}{W(x)}\le -\frac{\langle x, \nabla V_\nu(x)\rangle}{|x|} 1_{\{|x|\ge 1\}}+ b1_{\{|x|\le 1\}} $$ for some constant $b>0$. Then, the required assertion follows from Lemma \ref{L3.5} and the proof of Theorem \ref{T3.1}(1).

In case (2), we consider a smooth function such that $W(x)=\e^{(1-\delta)V(x)}$ for $|x|\ge 1$ and $W(x)\ge 1$ for all $x\in \R^d$. Then, $$\frac{L_{\nu}W(x)}{W(x)}\le -(1-\delta)\big(\Delta V(x)-\delta|\nabla V(x)|^2\big)+ b1_{\{|x|\le 1\}} $$ for some constant $b>0$. This along with Lemma \ref{L3.5} and the proof of Theorem \ref{T3.1}(2) also yields the desired assertion.  \end{proof}

According to the proof of Corollary \ref{C3.2}, when the measure $\nu$ has the compact support, we can obtain the following statement from Theorem \ref{T3.4}.

\begin{cor}\label{C3.6} Let $\nu$ be a probability measure
on $\R^d$ with compact support such that $R:=\sup\{|z|: z\in {\rm
supp}\,\nu\}<\infty.$  \beg{enumerate}
\item[$(1)$] If \begin{equation}\label{C3.6.1}\liminf_{|x|\to\infty}
\ff{\<\nn V(x),x\>-R|\nn V(x)|}{|x|}=\infty,\end{equation} then
$\mu*\nu$ satisfies the super Poincar\'{e} inequality
\eqref{sup-pin} with $$\beta(r)=c \Big(1+\alpha(\psi(2/r),
r/2)\Big) $$ for some constant $c>0$, where
$$\psi(r):=\inf\bigg\{s>0: \inf_{|x|\ge 2s}
\ff{\<\nn V(x),x\>-R|\nn V(x)|}{|x|}\ge r\bigg\}.$$
\item[$(2)$] If there is a constant $\delta\in(0,1)$ such that
  \begin{equation}\label{C3.6.2} \liminf_{|x|\to\infty}  \big(\delta|\nabla V(x)|^2-\DD V(x)\big)=\infty,\end{equation}
   \end{enumerate} then $\mu*\nu$ satisfies the super Poincar\'{e} inequality \eqref{sup-pin} with
   $$\beta(r)=c \Big(1+\alpha(\tt\psi(2/r), r/2)\Big) $$ for some constant $c>0$, where
$$\tt\psi(r):=\inf\bigg\{s>0: \inf_{|x|\ge 2s}
\big(\delta|\nabla V(x)|^2-\DD V(x)\big)\ge r\bigg\}.$$
\end{cor}

The proof of Corollary \ref{C3.6} is similar to that of Corollary \ref{C3.2}, and is thus omitted.
 Finally,   we consider the following example to illustrate   Corollary \ref{C3.6}.

\begin{exa}\label{E3.7} Let $V(x)=c+|x|^p$ for some $p>1$ and $c\in\R$ such that $\mu(\d x):= \e^{-V(x)}\d x$ is a probability measure on $\R^d$.  Then for any compactly supported probability measure $\nu$, there exists a constant $c>0$ such that $\mu*\nu$ satisfies the super Poincar\'{e} inequality \eqref{sup-pin} with
\beq\label{WFJ} \beta(r)=\exp\Big(cr^{-\frac{p}{2(p-1)}}\Big),\ \ r>0.\end{equation}
\end{exa}

\begin{proof}  Since by \cite[Corollary 1.2]{W13} the super Poincar\'e inequality implies the Poincar\'e inequality, we may take $\bb(r)=1$ for large $r>0$. So, it suffices to prove the assertion for small $r>0.$                                                                                                     As explained in the proof of Example \ref{E3.3}
up to a bounded perturbation, we may simply assume that $V\in C^2(\R^d)$.   For any $\delta\in(0,1)$ and any $x\in\R^d$ with $|x|$ large enough,
$$\delta|\nabla V(x)|^2-\DD V(x)\ge \eta (V(x)),$$
where $\eta$ is a non-decreasing function such that $\eta(r)=\delta r^{2(p-2)/p}$ for some constant $\dd>0$ and all $r\ge 1$.  So,
$$\tt \psi(u) \le c_1 \big(1+u^{\ff 1 {2(p-2)}}\big),\ \ \ u>0$$ holds for some constant $c_1>0$. Next, it is easy to see that
$$\aa(r,s) \le c_2(1+s^{-d/2}) \e^{c_2 r^p},\ \ s,r>0$$ holds for some constant $c_2>0$. Therefore, the desired assertion for small $r>0$ follows from   Corollary \ref{C3.6}(2).  \end{proof}

\paragraph{Remark 4.1}  (1) By letting $\nu=\dd_0$ we have $\mu=\mu*\nu$. So,   Example \ref{E3.7} implies that
$\mu$ satisfies the super Poincar\'e inequality with $\bb$ given in \eqref{WFJ} for some constant $c>0$, and moreover,  the inequality   is stable under convolutions of compactly supported probability measures. It is easy to see from \cite[Theorem 6.2]{W00a} that the rate function $\bb$ given in \eqref{WFJ} is sharp, i.e. $\mu*\nu$ does not satisfy the super Poincar\'e inequality with $\bb$  such that $\lim_{r\downarrow 0} r^{\ff p{2(p-1)}} \log \bb(r)=0.$

(2) On the other hand, however, if $\nu$ has worse concentration property, $\mu*\nu$ may only satisfy a weaker functional inequality. For instance,
let $\mu$ be in Example \ref{E3.7} but $\nu(\d z)= C\e^{-|z|^q}\d z$ for some constant $q\in (1,p)$ and normalization constant $C>0$. As explained in Remark 4.1(1) for $q$ in place $p$ we see that $\nu$ satisfies the super Poincar\'e inequality with
\beq\label{DGG} \beta(r)=\exp\Big(cr^{-\frac{q}{2(q-1)}}\Big),\ \ r>0\end{equation}  for some constant $c>0$.  Combining this with the super Poincar\'e inequality for $\mu$ with $\bb$ given in \eqref{WFJ}, from Proposition \ref{P1.1} we conclude that $\mu*\nu$ also satisfies the super Poincar\'e inequality with $\bb$ given in \eqref{DGG} for some (different)  constant $c>0$, which is sharp   according to \cite[Theorem 6.2]{W00a} as explained above.   However, it is less straightforward to verify this super Poincar\'e inequality for $\mu*\nu$ using Theorem \ref{T3.4} instead of Proposition \ref{P1.1}.

\paragraph{\bf Acknowledgement.} We would like to thank  the referee for helpful comments.

\end{document}